\documentclass{article}
\usepackage{graphics}
\usepackage{ifthen}
\usepackage{amsmath,amsfonts,amstext,amscd,bezier,amssymb,a4,enumerate,epsfig,theorem,psfrag,subfigure}
\usepackage[latin1]{inputenc}
\usepackage{graphicx}

\newcommand{\RR}{{\mathbb{R}}}
\newcommand{\CC}{{\mathbb{C}}}

\newcommand{\cJ}{{\cal{J}}}

\newcommand{\cL}{{\cal{L}}}

\newcommand{\cM}{{\cal{M}}}
\newcommand{\cO}{{\cal{O}}}

\newcommand{\cP}{{\cal{P}}}

\newcommand{\Ss}{{\cal{S}}}
\newcommand{\cS}{{\cal{S}}}
\newcommand{\Tt}{{\cal{T}}}
\newcommand{\cT}{{\cal{T}}}

\newcommand{\Uu}{{\cal{U}}}
\newcommand{\cU}{{\cal{U}}}

\newcommand{\ILa}{\Tt_\Lambda}

\newcommand{\optpar} [2] []
  {\ensuremath{#2\ifempty{#1} {} {\!\left(#1\right)}}}
\newcommand{\optsub} [2] []
  {\ensuremath{#2\ifempty{#1} {} {_{#1}}}}
\newcommand{\optsuper} [2] []
  {\ensuremath{#2\ifempty{#1} {} {^{#1}}}}

\newlength{\boxwd}
\newlength{\boxht}

\newcommand{\ifempty} [1] {\ifthenelse{\equal {#1} {} }}

\newcommand{\UpLa}{\Uu^\pi_\Lambda}

\newtheorem{theorem}{Theorem}[section]
\newtheorem{corollary}{Corollary}
\newtheorem{lemma}[theorem]{Lemma}
\newtheorem{proposition}{Proposition}

\theoremstyle{definition}

\newcommand\capsize{\relax}

\newcommand\diag{\operatorname{diag}}

\newcommand\tr{\operatorname{tr}}

\title{The Toda lattice, old and new}
\author{Carlos Tomei \vspace{0.5cm} \\
Departamento de Matemática, PUC-Rio\footnote{ R. Mq.  S. Vicente 225,
Rio de Janeiro 22451-900,  Brazil. email: tomei@mat.puc-rio.br. The author gratefully acknowledges support from CAPES, CNPq and FAPERJ, Brazil.} \vspace{0.5cm} \\
}
\date{}

\begin{document}
\maketitle

\begin{abstract} Originally a model for wave propagation on the line, the Toda lattice is a wonderful case study in mechanics and symplectic geometry. In Flaschka's variables, it becomes an evolution given by a Lax pair on the vector space of real, symmetric, tridiagonal matrices. Its very special asymptotic behavior was studied by Moser by introducing norming constants, which play the role of discrete inverse variables in analogy to the solution by inverse scattering of KdV. It is a completely integrable system on the coadjoint orbit of the upper triangular group. Recently,  bidiagonal coordinates, which parameterize also non-Jacobi tridiagonal matrices, were used to reduce asymptotic questions to local theory. Larger phase spaces for the Toda lattice lead to the study of isospectral manifolds and different coadjoint orbits. Additionally, the time one map of the associated flow is computed by a familiar algorithm in numerical linear algebra.

The text is mostly expositive and self contained, presenting alternative formulations of familiar results and applications to numerical analysis.
\end{abstract}

\medbreak
{\noindent\bf Keywords:} {Toda lattice, completely integrable systems, QR algorithm, isospectral manifolds, inverse scattering.}

\medbreak
{\noindent\bf MSC-class:}{Primary: 65F15, 37S35; Secondary: 53D05.}

\section{Introduction}
The Toda lattice is a beautiful mathematical object: mathematical miracles and serendipity are abundant. Somehow, everybody has something to say about it, and this extends way beyond mathematicians: the connections with physics and numerical analysis are fruitful and clarifying. Moreover, the formalism is  versatile, accomodating a number of interesting dynamical systems.

In this text, we present some of these many aspects of the Toda lattice. After more than thirty years of intriguing discoveries, the subject is not exhausted even in its most elementary formulation. We aim at concrete examples, and provide occasional pointers to more abstract approaches.

In Section 2, the differential equation is presented as the physical system introduced by Toda and converted into an evolution on Jacobi matrices by Flaschka's remarkable change of variables. The two basic properties are discussed: the conservation of eigenvalues and the very simple asymptotic behavior. A comparison with the archetypical Lax pair, the Korteweg-de Vries equation, motivates the introduction of Moser's inverse variables on Jacobi matrices, given by eigenvalues and (discrete) norming constants.

The change of variables taking the original mechanical system in $\RR^{2n}$ to the set of Jacobi matrices with a fixed trace $\cJ_0$ leads to a brief description of the Hamiltonian formalism. After a brief review of completely integrable systems and coadjoint orbits as symplectic spaces, we interpret the Toda equation as a Hamiltonian system on coadjoint orbits (in particular $\cJ_0$), as done originally by Adler and Kostant. We then consider larger phase spaces: generic coadjoint orbits and isospectral manifolds of tridiagonal matrices.

Section 4 is another geometric approach to the Toda lattice, which might have anticipated its discovery by decades. It leads very naturally to the solutions by factorization by Symes. Section 5 is dedicated to an interpolation theorem which connects the subject to numerical analysis: one Hamiltonian in the completely integrable collection associated to the Toda equation gives rise to a flow which, at integer times, consists of matrices obtained by the so called $QR$ algorithm, extensively used in algorithms to compute eigenvalues.

Norming constants have a drawback: they do not extend to the boundary of the set of Jacobi matrices, which is where the fine asymptotic properties of the flows occur. Section 6 describes bidiagonal coordinates, which provide charts for the full isospectral manifold of tridiagonal matrices. A brief interlude on $QR$ algorithms with shifts is provided as an example of the versatility of this new instrument.

As shown above, the content of this review is biased. Indeed, a short text cannot provide coverage for all the ramifications of the subject. As a minimal list of alternatives, the curious reader is invited to consider \cite{Pe} and \cite{RSTS} for algebraic aspects related to integrability, \cite{LT} for uses of symplectic geometry to the analysis of (variations of) the original Toda system, \cite{KS} for an essentially orthogonal overview and \cite{R} for a relativistic mutation, one of many interesting physical directions from our starting point.

\section{Physical origins and Flaschka's variables}

In the mid sixties, the Japanese physicist M. Toda \cite{Tod} proposed  a model for wave propagation along $n$ particles in a line by the Hamiltonian
\[ H(x,y) = \frac{1}{2} \sum_1^n y_k^2 + \sum_1^{n-1} \exp(x_k - x_{k+1}). \]
As usual, positions $x_k$  and velocities $y_k$ vary as
\[\dot{x}_k = \frac{\partial H}{\partial y_k} = y_k, \quad \dot{y}_k = - \frac{\partial H}{\partial x_k}= \exp(x_{k-1}-x_k) - \exp(x_k - x_{k+1}),\ i=1,\ldots,n.\]
Here, $x_0 = -\infty, x_{n+1} = +\infty$ --- this is the \textit{non-periodic} case of the {\it Toda lattice}. The particles are labeled and from a mathematical viewpoint, there is nothing wrong with occasional collisions: they simply pass each other.

Clearly, the energy $H(x(t), y(t))$ is a conserved quantity. The invariance under translation in configuration space (i.e., the fact that $H(x,y) = H(x + c, y)$, for any fixed $c \in \RR^n$) implies that the center of mass of the system moves uniformly (i.e., the linear momentum of the system is preserved).

Following Flaschka\cite{Fl},
shift to center of mass coordinates (i.e., work with differences $x_k - x_{k+1}$), get rid of exponentials and adjust constants,
$$a_k = -\frac{y_k}{2} \, , \quad k=1,\ldots,n,\quad
b_k = \frac{1}{2}\  e^{(x_k - x_{k+1})/2}, \quad k=1,\ldots,n-1$$
and the evolutions for $x(t)$ and $y(t)$ become
\[ a_k ' = 2 (\, b_k^2 - b_{k-1}^2 \,), \, k=1,\ldots,n, \quad b_k'= b_k(\, a_{k+1} - a_k\, ), \, k=1,\ldots,n-1\]
which, surprisingly, can be cast as the matrix differential equation
$$\dot{J} = [J,\Pi_{sk} J] = J (\Pi_{sk} J) - (\Pi_{sk} J) J. \eqno(J)$$
Here, $J$ is a Jacobi matrix with diagonal entries $a_k$ and principal off-diagonal entries $b_k$. A {\it Jacobi matrix} $J$ is a real $n \times n $ symmetric matrix which is tridiagonal (i.e., $J_{ij}=0$ if $|i-j|>1$) and such that the principal off-diagonal entries (those for which $|i-j|=1$) is a positive number.
The matrix $\Pi_{sk} M$ is skew-symmetric, with the same lower triangular part as $M$.

There is nothing wrong in considering equation $(J)$ for arbitrary real symmetric matrices (and even non-symmetric matrices). Explicitely, consider the differential equation on real symmetric matrices $S(t)$ given by
\[ S'(t) = [ S(t) , \Pi_{sk} S(t) ], \quad S(0) = S_0. \eqno(S) \]

\subsection{Lax pairs, asymptotic behavior}

We present two fundamental facts. The first one is a question about differential equations. What kind of differential equation on matrices gives rise to an evolution $M(t)$ which preserves the eigenvalues? More precisely, what kind of vector field gives rise to a flow of the form $M(t) = ( P(t) )^{-1} M(0) P(t)$? Simply take derivatives,
\[ M'(t) = - ( P(t) )^{-1} P'(t) ( P(t) )^{-1} M(0) P(t) + (P(t) )^{-1} M(0) P'(t) \]
\[= - ( P(t) )^{-1} P'(t) M(t) + M(t) ( P(t) )^{-1} P'(t) =
[ M(t), ( P(t) )^{-1} P'(t)].\]

In words, equations given by {\it Lax pairs} $M'(t) = [M(t), X(t)]$ are solved by conjugating the initial condition. In particular, if the initial condition is a real, symmetric matrix $S(0)$, and we require $P$ to be an orthogonal matrix (so that symmetry, together with spectrum, is preserved along the orbit $S(t)$), one should consider a Lax pair of the form
$S'(t) = [ S(t), A(t)]$, where $A(t)$ is a real, skew symmetric matrix (this follows from the standard computation --- take the derivative at $0$ of a curve of orthogonal matrices $Q(t)$ with $Q(0)=I$).

\begin{proposition} \label{spectrum}
The solution $S(t)$ of $(S)$ starting from a real, symmetric matrix $S(0)$ is of the form $S(t) = Q(t)^\ast S(0) Q(t)$: it is well defined for all $t \in \RR$.
\end{proposition}

An explicit form of $Q(t)$ is given in Section \ref{factorizations}.

\begin{proof} We are left with showing global existence: the matrix norm $ ||S||^2 = \tr S^T S$ is conserved along an orbit.
\end{proof}

We now consider the asymptotic behavior for $S(t)$ and $J(t)$, the orbits starting from arbitrary symmetric and Jacobi matrices. In the Jacobi case, we may get information from the physical interpretation: informally, minimizing the potential energy (the second term in the Hamiltonian) leads to the spreading of the particles, so that one might expect $x_k - x_{k+1} \to - \infty$. Thus, for $t \to \infty$, particles should be essentially independent from each other and undertake uniform motion. In Flaschka's variables, the orbit $J(t)$ should converge to a diagonal matrix.

\begin{proposition} \label{asymptotics}
The limits for $t \to \pm \infty$ of $S(t)$ are diagonal matrices. When $t \to \infty$ (resp. $t \to - \infty$), $J(t)$ converges to a diagonal matrix with strictly decreasing (resp. increasing) entries along the diagonal.
\end{proposition}

\begin{proof} The differential equations for the diagonal entries of $S(t)$ are given by
\[ S_{11}' = 2( S_{12}^2 + S_{13}^2 + \ldots + S_{1n}^2), \]
\[ S_{22}' = 2( -S_{21}^2 + S_{23}^2 + \ldots + S_{2n}^2), \]
and, in general, $S_{kk}' = - \sum_{j<k} S_{kj}^2 + \sum_{j>k} S_{kj}^2$. The equation for $S_{11}'$ indicates that it is nondecreasing in time. The same is true not for $S_{22}'$ but for $S_{11}' + S_{22}'$, and in general the partial traces $\sum_{j=1}^k S_{jj}$ are nondecreasing. Since the norm of $S(t)$ is constant, all entries $S_{ij}$ and the partial traces (and their derivatives) are uniformly bounded. In particular, the derivatives of the partial traces are integrable, Lipschitz functions on $\RR$ and hence all entries $S_{ij}, i \ne j$ must go to zero. This implies the  diagonal convergence of $S(t)$ (and hence of $J(t)$) at $\pm \infty$.

From equation $b_k'= b_k(\, a_{k+1} - a_k\, )$, the  entries $b_k(t)$ of $J(t)$ are always positive (indeed, if $b_k=0$ at some time, it is always zero). When $t \to \infty$, we must then have $a_{k+1} - a_k <0$, so that the diagonal limit matrix has entries in strictly decreasing order (the fact that the eigenvalues of a Jacobi matrix are all distinct is proved in Lemma \ref{simplicity}). A similar argument obtains the result for $t \to -\infty$.
\end{proof}

As a side remark, notice that the above proposition implies the spectral theorem (only in finite dimensions: the presence of continuous
spectrum brings up new issues to the asymptotics of the Toda flow \cite{DLT3}). The interested reader might enjoy a proof  of the {\it Wielandt-Hoffman} theorem along similar lines (\cite{DRTW}).

\subsection{Scattering and inverse variables} \label{LaxMoser}

Flaschka was certainly inspired by P. Lax's
casting of the Korteweg-de Vries (KdV) equation into \textit{Lax pair} form.
Briefly, the evolution
\[ u_t(t,x) = 6 \ u(t,x) \ u_x(t,x) - u_{xxx}(t,x), \qquad t \ge 0, \ x \in \RR \]
is equivalent to the operator evolution
\[ L'(t) = [L(t), H(t)] \ , \]
where
\[L(t) f(x) = - f''(x) + u(t,x) f(x), \]
\[H(t) f(x) = 4 f'''(x) - 6 u(t,x) f'(x) - 3 u_x(t,x) f(x). \]

Lax then reproved the seminal discovery of Gardner, Greene, Kruskal and Miura (\cite{GGKM}) that the the evolution of KdV varies the operator $L$ preserving its eigenvalues.
Analogous computations replacing $L$ by higher order operators lead to the \emph{ Gelfand-Dickey flows},  which include the Boussinesq equation and more (\cite{BDT}).

The connection between differential equations like KdV and the {\it inverse scattering method} discovered in the late sixties (\cite{GGKM}, \cite{AKNS}) led to intense research on integrable systems. Originally as a formalism, inverse data provided (infinite dimensional) action-angle variables for the KdV equation.

Inverse variables for Jacobi matrices were pervasive in the early approaches to the Toda flows, starting with Moser (\cite{M}). The theorem below, that he attributed to Stieltjes,  provides inverse variables for the Toda equation $(J)$ (\cite{PT}, \cite{KP}).

\begin{lemma} \label{simplicity}
Jacobi matrices always have distinct real eigenvalues. The first and last coordinate of each eigenvector is a nonzero number.
\end{lemma}

\begin{proof} Suppose $J v = \lambda v$ and $v[1]=0$: equating the first coordinates in the equation obtains $v[2]=0$, and successively --- thus, the extreme entries of eigenvectors are nonzero. Now, two independent eigenvectors associated to the same eigenvalue  would yield an eigenvector with first coordinate equal to zero.
\end{proof}

\begin{theorem} \label{normingconstants}
The set of Jacobi matrices is diffeomorphic to the product
\[ \{ (\lambda_1, \lambda_2,\ldots,\lambda_n) \, ; \lambda_1 < \lambda_2 <\ldots < \lambda_n \} \times \{(v_1, v_2, \ldots, v_n) \, ; v_i >0, \  \sum_i v_i^2 = 1 \}. \]
More precisely, the diffeomorphism takes $J$ to its ordered eigenvalues and to the first coordinates of its associated eigenvectors, normalized so as to be positive.
\end{theorem}

Parlett  made the intriguing observation that this algorithm is used by numerical analysts to obtain a tridiagonal matrix from a full symmetric matrix, not from a diagonal matrix.

\begin{proof} We sketch a procedure to invert this map presented in \cite{DNT}. Define the matrix $\Lambda$ and the vector $v$ as
\[ \Lambda= \diag(\lambda_1 , \lambda_2 , \ldots , \lambda_n)\ , \quad v = (v_1, v_2, \ldots, v_n) \]
and consider the sequence of vectors $v, \Lambda v, \ldots, \Lambda^{n-1} v$. The vectors $q_1 = v, q_2, \ldots, q_n$ obtained from this sequence by applying the Gram-Schmidt method are the columns of an orthogonal matrix $Q$ for which $J = Q^T \,\Lambda\, Q$.
\end{proof}

The entries of the vector $v$ are the {\it norming constants} associated to $J$.
Still in analogy with KdV, as $J(t)$ solves equation $(J)$, the eigenvalues stay put, as Flaschka knew, and Moser showed that the norming constants $v(t)$ varied in a simple fashion:
simply normalize (under the Euclidean norm) the vector $exp(t\Lambda)\ v(0)$.

From Proposition \ref{asymptotics}, at $t= \pm \infty$ Jacobi orbits $J(t)$ converge to diagonal matrices with ordered eigenvalues along the diagonal entries.
Using norming constants, Moser computed the {\it scattering map} of the Toda flow. Particles group in pairs with the same asymptotic velocity at extreme times:  the quantity of interest is the {\it shift}, the distance between the two straight lines tangent to the asymptotic motion at $\pm \infty$ of particles in the same pair.

Norming constants, alas, do not parameterize the limit matrices of the flows. Actually, they degenerate on a large part of the boundary of the set of Jacobi matrices. This will be circumvented by the {\it bidiagonal coordinates} in section \ref{bidiagonal}.

There is a basic issue which has not been handled carefully so far. Equation $(J)$, given in terms of skew symmetric matrices $A(t)$, was shown to imply that the solution is an orthogonal conjugation of the initial condition $J(0)$ --- thus, in particular, $J(t)$ is always real, symmetric, and eigenvalues are preserved. But why should the evolution preserve the tridiagonal form? This has to happen, if Flaschka's change of variable preserves the physical meaning of the variables. Also, this fits with the  solution by inverse variables of $(J)$. A more conceptual argument showing that $J(t)$ is always a Jacobi matrix will be presented in the next section. Yet another argument will  come up in Section \ref{factorizations}.

\section{Some symplectic geometry}

From its physical description in terms of positions $x_i$ and velocities $y_i$, it is clear the Toda lattice admits a Hamiltonian formulation in $\RR^{2n}$. By Flaschka's change of variables, after removal of the (trivial) evolution of the center of mass, the phase space for the differential equation becomes the set of Jacobi matrices with trace equal to zero, a cone of dimension $2n-2$, from Theorem \ref{normingconstants}. Somehow, one should be able to transfer the original Hamiltonian formulation to the new variables, and still proceed with the study of the Toda lattice as a problem in mechanics within the new phase. This indicates a more general context, the starting point of a vast field,  {\it symplectic geometry}. In the next subsection, we outline some basic requisites for this project.

\subsection{Complete integrability}

We start with a brief description of a more general definition of {\it Hamiltonian formulation} of a vector field. Take a differential equation with phase space $M$, or more precisely, consider the associated vector field $Z$ defined on the tangent bundle $TM$. First equip $M$ with a  closed, nondegenerate, $2$-form $\omega$: the pair $(M,\omega)$ is a \textit{symplectic manifold}, necessarily of even dimension, say $2n$ (excellent sources for symplectic geometry are \cite{AM}, \cite{GS}).

Each \textit{Hamiltonian} $H: M \to \RR$ gives rise to a vector field $X_H$ as follows: the contraction of $\omega$ with $X_H$ should obtain the 1-form $dH$. Said differently, for every vector field $Y$ in $M$, at each $m \in M$
\[ \omega_m(X_H(m), Y(m)) = dH(Y(m)) = Y(H)(m), \]
where $Y(H)(m)$ is the derivative of $H$ along $Y$ at $m$. The vector field $Z$ admits a Hamiltonian formulation if $Z = X_H$ for an appropriate choice of $\omega$ and $H$.

The simplest example is the standard 2-form in $\RR^{2n} = \{(x,y); \, x, y \in \RR^n\} $ given by $\omega = \sum_k x_k \wedge y_k$. Let us find the vector field $X_H =\sum_k A_k \partial _{x_k} + \sum_k B_k \partial _{y_k}$ associated to the Hamiltonian $H: \RR^{2n} \to \RR$ (here, $\partial_z$ is the partial derivative in the variable $z$, and $H_z = \partial_z H$). For an arbitrary vector field $Y = \sum_k a_k \partial _{x_k} + \sum_k b_k \partial _{x_k}$,
\[ \omega(X_H, Y) = \sum_k - a_k A_k + \sum_k b_k B_k =  \sum_k a_k H_{x_k} + \sum_k b_k H_{y_k} = dH(Y) \, , \]
and, as is well known, $X_H = \sum_k H_{y_k}  \partial_{x_k} - \sum_k H_{x_k} \partial_{y_k}$.

When is a Hamiltonian vector field $X_H$  \textit{completely integrable}?
Complete integrability requires $n$ {\it commuting}  Hamiltonians $H_k: M \to \RR$ (i.e., their induced vector fields $X_k = X_{H_k}$ commute, or equivalently, such that $\{ H_i, H_j\}=0$ for the {\it Poisson bracket} induced by the 2-form $\omega$, $\{ H_i, H_j\} = \omega(X_{H_i}, X_{H_j})$ ) among themselves and with $H$ itself. Finally, the Hamiltonians $H_k$ should be {\it functionally independent} (i.e., their gradients should be linearly independent on a dense set of $M$).

Very few dynamical systems are completely integrable, but these are especially important for being situations in which explicit computations may be performed. Indeed, one can make precise the idea that a generic Hamiltonian vector field does not have a second conserved quantity (recall that $H$ itself is conserved along orbits). But this is just the opposite of what we expect from certain iterations in numerical analysis: to compute eigenvalues, for example, we expect to change something (an original matrix, an approximation of an eigenvector) without varying the objects being computed (the eigenvalues themselves).

The \textit{Liouville-Arnold-Jost theorem}  states that, under appropriate
hypothesis, the phase space of completely integrable equations foliates into invariant
tori (i.e., products of lines and circles),
given by levels of the conserved quantities $H_k$ (frequently called the \textit{action variables} in this context).
In each torus, in  \textit{angle variables}, the evolution is just straight line motion.
The angles vary smoothly at neighboring tori, and the global dynamics is mostly dependent
on the arithmetic properties of the angular velocities.

\subsection{Toda flows in coadjoint orbits} \label{coadjoint}

Using Flaschka's change of variables, one might push forward the standard 2-form in
$\RR^{2n}$ to ${\mathcal J}_0$, the cone of Jacobi matrices with zero trace, thus converting it into a symplectic manifold. The surprising fact is that the resulting 2-form comes up from another construction of great interest, which we now describe briefly.

A large class of symplectic manifolds is given by {\it coadjoint orbits}, equipped with the {\it Lie-Kirillov-Poisson} 2-form. Let $G$ be a (finite dimensional) \, Lie group, $\mathfrak G$ its Lie algebra, and identify ${\mathfrak G}^\ast$, the dual of the Lie algebra, by means of a {\it nondegenerate coupling}
\[ (\alpha, A) \in {\mathfrak G}^\ast \times {\mathfrak G} \to \langle \alpha , A \rangle \in \RR. \]
The coupling is bilinear and nondegeneracy means that the restrictions
$\langle \alpha, . \rangle$ and $\langle . , A \rangle$, for $\alpha, A \ne 0$, give rise to nonzero functionals respectively on ${\mathfrak G}^\ast$ and ${\mathfrak G}$.
Thus, all linear functionals in ${\mathfrak G}$ are of the form $A \mapsto \langle \alpha, A  \rangle$ for some  $\alpha \in {\mathfrak G}^\ast$.

For the Toda flow, start with $G = U^+$, the group of $n \times n$ upper triangular real matrices with positive diagonal entries. Then $\mathfrak G= \mathfrak U$ is the vector space of real upper triangular matrices and we may identify  ${\mathfrak G}^\ast$ with $\cS$, the vector space of real, symmetric matrices, through the pairing $ \langle S, A\rangle = \tr S A$.

The group $G$ acts on itself by conjugation and on its Lie algebra $\mathcal G$ by its derivative at the origin, {\it the adjoint action}. For $G = U^+$, the adjoint action is given by $( g, U) \in G \times {\mathfrak G} \mapsto Ad_g(U) = g U g^{-1} \in {\mathfrak G}$. The {\it coadjoint action} from $G \times {\mathfrak G}^\ast$ to ${\mathfrak G}^\ast$
satisfies the natural duality,
\[ \langle Ad^\ast_g (\alpha) , A \rangle = \langle \alpha ,  Ad_{g^{-1}} ( A) \rangle, \quad \forall \, \alpha \in {\mathfrak G}^\ast, \forall \, A \in {\mathfrak G}.\]

Back to the Toda context,
\[ \forall \, \alpha \in \cS, \forall \, A \in \mathfrak U \, \langle,  Ad^\ast_g \alpha , A \rangle = \langle \alpha ,  Ad_{g^{-1}} ( A) \rangle = \tr \alpha  {g^{-1}}  A g = \tr (g \alpha  {g^{-1}})  A .\]
The bad news is that $g \alpha  {g^{-1}}$ is not a symmetric matrix.
Denote by $\mathfrak {sU}$ the vector space of real, strictly upper triangular matrices. Clearly, for $E \in {\mathfrak {sU}}$ and $A \in {\mathfrak U}$ we have $\tr E A =0$.  Consider the (unique) splitting
\[ M = S + E = \Pi_{\cS} M + \Pi_{\mathfrak {sU}} M, \quad S \in \cS, E \in {\mathfrak {sU}}.\]
We clearly have
\[\forall \, A \in \mathfrak U \ \langle Ad^\ast_g \alpha , A \rangle = \tr (g \alpha  {g^{-1}})  A = \tr \Pi_{\cS} (g \alpha  {g^{-1}})  A, \]
from which we finally obtain $Ad^\ast_g \alpha = \Pi_{\cS} (g \alpha  {g^{-1}})$.

By definition, $M$ and $\Pi_{\cS} M$ have the same lower triangular part (and this includes the diagonal). In particular,  as observed by Adler \cite{Ad} and Kostant \cite{K}, Jacobi matrices with fixed trace form a {\it coadjoint orbit}.

We now recall the celebrated Lie-Kirillov-Poisson 2-form on a coadjoint orbit. Let $\cO_\alpha$ be the coadjoint orbit through $\alpha$. Any vector $a$ in the tangent space of $\cO_\alpha$ at $\alpha$ is the derivative at zero of $Ad^\ast_{exp(tA)} \alpha$ fro some $A$ in the Lie algebra ${\mathfrak G}$.  Let $A$ and $B$ in ${\mathfrak G}$ give rise to  tangent vectors $a$ and $b$. Set
\[ \omega_\alpha( a, b ) =  - \langle \alpha, [ A, B] \rangle. \]
The minus sign is  innocuous, but there is so much to prove here. First, it is not clear that $\omega$ is well defined: other elements in $\mathfrak G$ might give rise to the same tangent vectors at $\alpha$. More, $\omega$ has to be proven nondegenerate and closed. We will take all those issues for granted.

Instead, we continue with the computations related to the Toda flow.
In this case, the curve $Ad^\ast_{exp(tA)} \alpha = \Pi_\cS \ e^{tA} \alpha e^{-tA}$ has the tangent vector $\Pi_\cS [A, S]$ at $\alpha$. The 2-form is given by
\[ \omega_S ( \Pi_\cS [A, S], \Pi_\cS [B, S] ) = - \tr S [A, B].\]

A simple computation shows that the initial Hamiltonian $H(x,y)$ for the Toda lattice in physical variables converts to $H(S) = \tr S^2/2$ in Flaschka' s variables.
\begin{proposition} \label{coadjointToda}
The Toda lattice is the vector field $X_H$ associated to the Hamiltonian $H(S) = \tr S^2/2$ defined on a coadjoint orbit $\cO_S$.
\end{proposition}

\begin{proof}
We search for a vector field $X_H = \Pi_\cS [U, S]$ for which, given any vector field $Y = \Pi_\cS [V, S]$, we must have
\[ \omega_S(X_H, Y) = - \tr S [U,V] = Y(H) = -\tr S \Pi_\cS [V, S]. \]
Let $\mathfrak A$ be the vector space (Lie algebra!) of real, skew symmetric matrices. We consider another splitting,
\[ M = A + U = \Pi_{sk}\,  M + \Pi_{up} M, \quad A \in {\mathfrak A}, U \in {\mathfrak U} \]
and make use of the orthogonalities ${\mathfrak A} \perp \cS$ and ${\mathfrak U} \perp {\mathfrak {sU}}$:
\[ \tr S \ [U,V] =  \tr S \ \Pi_\cS [V, S] =  \tr \Pi_{up} S \ \Pi_\cS [V, S] \]
\[= \tr \Pi_{up} S  \ [V, S] = \tr [\ S,\ \Pi_{up} S ] \ V= -\tr [\ S,\ \Pi_{sk}\,  S ] , \]
so that $-\tr [S, U] \ V = \tr [\ S,\ \Pi_{sk}\,  S ] \ V$, for all $V \in {\mathfrak U}$, which is equivalent to say that $[S,U] = [\ S,\ \Pi_{sk}\,  S ] + E$, for some $E \in {\mathfrak {sU}}$, so that
\[ X_H = \Pi_\cS [S,U ]  = \Pi_\cS [\ S,\ \Pi_{sk}\,  S ] =  [\ S,\ \Pi_{sk}\,  S ] . \]
\end{proof}

We are thus led to a very geometric explanation for the fact that the solution of the equation $(J)$, $J'(t) = [ J(t), B (J(t))]= [ J(t), \Pi_{sk}\,  J(t)], $ is always a Jacobi matrix, if $J(0)$ is: it simply does not leave the coadjoint orbit $\cO_{J(0)}$.

A similar computation obtains the commutativity of the Hamiltonians
$H_k(J) = \tr J^k, k=1, \ldots,n$). Indeed, the computation above generalizes to yield the {\it Adler-Kostant-Symes} theorem (\cite{S1}), a criterion to identify families of commuting Hamiltonians.
For the Toda lattice on Jacobi matrices with trace given by the initial condition $J(0)$, the $H_k$'s form a family of $n-1$ commuting flows which are functional independent throughout $\cO_{J(0)}$, since Jacobi matrices have simple spectrum from Lemma \ref{simplicity}. Thus, the original Toda vector field and any other vector field given by a Hamiltonian of the form $H_f(J) = \tr f(J)$ are completely integrable.

From Theorem \ref{normingconstants}, the related invariant Liouville-Arnold tori, given by sets of $n \times n$ Jacobi matrices with fixed spectrum, is
diffeomorphic to the set of possible choices for norming constants --- the positive octant of the unit sphere in $\RR^{n-1}$.

\subsection{Toda flows in larger phase spaces} \label{larger}

From the computations above, the same Hamiltonian $H = \tr S^2 /2$ induces equation $(S)$ on larger coadjoint orbits and the proof of complete integrability on these phase spaces requires many more conserved quantities. The problem was considered in two cases. In \cite{DLNT}, the authors consider the orbits of maximal dimension given by  $2 [n^2/4]$ in$\cS$. The new commuting Hamiltonians $H(S)$ are obtained by {\it chopping}:
they are the (symmetric functions of the) roots of the determinants of the matrices obtained by
removing the first $k$ rows and last $k$ columns of the matrix $S - \lambda I$. In analogy to Moser's computations in the Jacobi orbit, angle variables are essentially the first components of (generalized) eigenvectors.

The generic phase space for real, nonsymmetric matrices
is handled in \cite{DLT1}. First, one needs to interpret the Toda equation as a Hamiltonian on an appropriate coadjoint orbit of dimension $n^2 - n$ of a Lie group given by a semidirect product. At a matrix $M$, the new commuting Hamiltonians are the coefficients of the polynomial $p(z,\lambda)=\det(M - z M^T - \lambda)$,
so a Riemann surface $p(z,\lambda)=0$ is associated to $M$ and is invariant under the Toda flow. The existence of such additional structure has been known since the first studies of the {\it periodic} Toda flow, where the particles move on a circle instead of in the line (\cite{VM}, \cite{Kr}, \cite{KrN}). For nonsymmetric matrices, the additional
angle variables are obtained from an extension of the Abel-Jacobi map,
by integrating a set of explicit meromorphic 1-forms on the surface along specific divisors, related again by generalized eigenvectors.

\subsection{The isospectral  manifold of tridiagonal matrices}

The set ${\mathcal J}_\Lambda$ of Jacobi matrices with a given simple spectrum
\[ \lambda_1 < \lambda_2 < \ldots < \lambda_n, \qquad \Lambda = \diag(\lambda_1, \lambda_2,\ldots, \lambda_n) \]
is diffeomorphic to $\RR^{n-1}$. There is a natural enlargement of this set --- define $\cT_\Lambda$, the  set of real, symmetric tridiagonal matrices with spectrum $\Lambda$.  This set is actually a compact manifold (\cite{T}).

\begin{figure}[h]
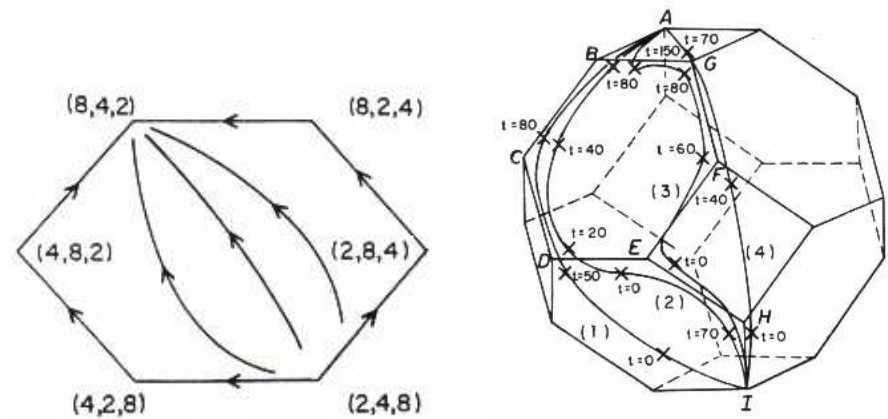

\centering
\mbox{\includegraphics[width=.45\linewidth]{hexa.eps}\quad\quad
\includegraphics[width=.45\linewidth]{cubocta.eps} }
\caption{$\bar{{\mathcal J}_\Lambda}$, for $n=3$ and $4$ } \label{polytopes}
\end{figure}

The first step in the understanding of $\cT_\Lambda$ is taking the closure (within the space of real,symmetric matrices) of ${\mathcal J}_\Lambda$. It turns out that $\bar{{\mathcal J}_\Lambda}$  has an interesting combinatorial structure, which we now describe.

For matrix dimensions $n=3$ and $n=4$ (and any choices of different eigenvalues) this set is homeomorphic to the polytopes in the figure. The vertices correspond to the $n!$ diagonal matrices with the same spectrum as $\Lambda$. For $n=3$, the diagram consists of matrices with eigenvalues 2, 4 and 8. Vertices are diagonal matrices, in a self-explanatory notation. Edges correspond to matrices having a single zero in the main off-diagonal entries. Thus the top edge consists of matrices with $(1,2)$ entry equal to zero; the eigenvalue $8$ is trapped at entry $(1,1)$ and the bottom $2 \times 2$ block consists of Jacobi matrices with eigenvalues $4$ and $2$, whose closure is homeomorphic to the whole edge. Edges are invariant under the Toda flow, and arrows indicate the sense of the flow. All interior points have the same $\alpha$ and $\omega$ limits.

For $n=4$ the boundary still consists of points with zero off-diagonal entries, which  split the matrix in two kinds of sets, eight of which are homeomorphic to the $3 \times 3$ case (hexagons), and six which are homeomorphic to the product of two $2 \times 2$ blocks with fixed spectrum, the quadrilaterals. Again, the curves represent some orbits of the Toda flow. For the general case (\cite{T}), \cite{BFR}), define the {\it permutohedron} $\cP_\Lambda= \hbox{conv}_{\pi \in S_n } \{(\lambda_{\pi(1)}, \ldots,\lambda_{\pi(n)})\}$.

\begin{theorem} The map from $\bar{{\mathcal J}_\Lambda}$ to the  permutohedron $\cP_\Lambda$
\[   J  \, =  \, Q^T \,\Lambda\, Q \, \mapsto \, \hbox{diag}(Q \,\Lambda\, Q^T) \in \cP_\Lambda \} \]
is a homeomorphism.
\end{theorem}

The existence of this homeomorphism in \cite{T} led
Bloch, Flaschka and Ratiu (\cite{BFR}) to search for a proof of the statement in terms of Atiyah's theorem on the convexity of the image of Hamiltonian torus actions (\cite{At}). Their proof starts with the identification of the homeomorphism as a momentum map, given by the formula above. A simple proof was later presented in \cite{LT}.

\begin{figure}[ht]
\centerline{\epsfig{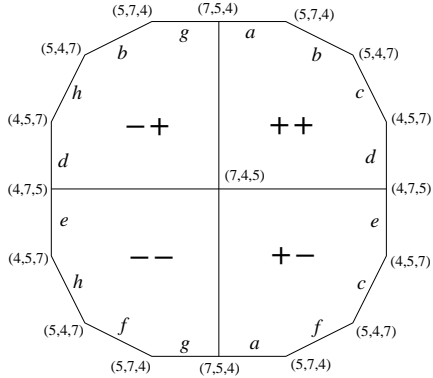}}
\caption{\capsize Gluing four copies of  $\bar{{\mathcal J}_\Lambda}$ }
\label{ilailah}
\end{figure}

From $\bar{{\mathcal J}_\Lambda}$ to the full isospectral manifold $\cT_\Lambda$, it's a game of mirrors (or, more precisely, the construction of an appropriate {\it Coxeter group} (\cite{T})). In a nutshell, dropping the signs of the off-diagonal entries of a real, symmetric tridiagonal matrix (preserving its symmetry!) does not change the eigenvalues --- this is something that numerical analysts use systematically: to compute eigenvalues of a matrix in
$\cT_\Lambda$, it suffices to handle Jacobi matrices. In particular, the sets of matrices in $\cT_\Lambda$ with nonzero entries split into $2^{n-1}$ connected components, all isomorphic to ${\mathcal J}_\Lambda$. To get $\cT_\Lambda$, one takes the closure of these components and glues them along faces which are naturally identified. For $n=3$ and $\Lambda= \{7,5,4\}$,  Figure \ref{ilailah} shows some edges already glued. Edges along the boundary have to be identified: boundary vertices must be the same, together with a sign (which?). The resulting bitorus is drawn in Figure \ref{bitorus} so as to emphasize the boundaries of the hexagons (more about this picture on  Section \ref{bidiagonal}).

\begin{figure}[ht]
\centerline{\epsfig{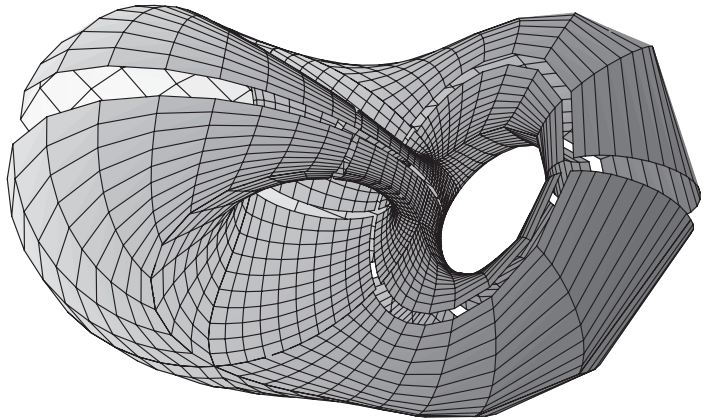}}
\caption{\capsize For $n=3$, $\cT_\Lambda$ is a bitorus}
\label{bitorus}
\end{figure}

The universal cover of $\cT_\Lambda$ is $\RR^{n-1}$ and the homology groups are torsion free (\cite{T}, \cite{Fr}). The sum of the partial traces, $F(S) = \sum_i (n-i+1) s_{ii}$ is a perfect Morse function: the equilibria are the diagonal matrices, and the Betti numbers are obtained by counting how many such matrices there are with a given signature.

\section{A missed opportunity}
We present another trail (\cite{LT}) leading to the Toda equation:  it might have anticipated the study of these equations by fifty years.

We fix notation. Let $\cM(n)$ consist of the $n \times n$ real matrices, $SO(n) \subset \cM(n)$ be the set of orthogonal matrices with positive determinant and $\Uu^+ (n) \subset \cM(n)$ be the set of upper triangular matrices with positive diagonal entries.

Clearly, both Lie groups act on $\cM(n)$ by conjugation.
For $M\in \cM(n)$,  set
\[ \cO_M^O = \{ Q^T M Q,\  Q \in SO(n) \}, \quad \cO_M^U = \{ R M R^{-1}, \ R \in \Uu^+ (n) \}\]
be respectively the {\it orthogonal} and {\it upper triangular} orbit through $M$. Clearly $\cO_M^O$ and  $\cO_M^U$ are connected manifolds of matrices with the same eigenvalues as $M$.

For convenience, let $M$ have  positive simple spectrum (i.e., positive distinct eigenvalues) --- we are interested in ${\mathcal I}_M = \cO_M^O \cap \cO_M^U$.

The dimension count might suggest that ${\mathcal I}_M$ is trivial, but this is not the case. Indeed, for a matrix $X \in \Ss$, there are $Q \in SO(n)$ and $R \in \Uu(n)$ such that
\[ X = Q^T M Q = R M R^{-1}  \ \Rightarrow \  (QR) \, M = M\, (QR). \]

In words, the matrix $QR$ commutes with $M$. Since $M$ has simple spectrum,  $QR$ is a function of $M$, in the sense that there exists $f : \RR \to \RR - \{0\}$ such that $QR = f(M)$. Notice that that  $Q$ and $R$  are real invertible matrices. This is the {\it $QR$ factorization} of $f(M)$ (see Appendix 1 for details). In general,  a real matrix $X$ with positive determinant admits a unique factorization $X =  [M]_Q [M]_R$, where $[M]_Q \in SO(n)$ and $[M]_R \in \Uu^+ (n)$ .

Adding up, given $f : \RR \to \RR ^+$, we obtain a matrix $\psi(f,M) \in {\mathcal I}_M$,
\[ \psi(f,M) = [ \ f(M) \ ]_Q ^T \ M \ [ \ f(M) \ ]_Q = [ \ f(M) \ ]_R \ M \ [ \ f(M) \ ]^{-1}_R \ .\]


\subsection{Commuting flows in ${\mathcal I}_M$}

The notation $\psi(f,M)$ was introduced to suggest that we are close to a flow on ${\mathcal I}_M$.

\begin{proposition} Let $f, g: \RR \to \RR^+$ and $M$ be a real matrix with real, simple spectrum. Then
\[ \psi(fg,M) = \psi(gf,M) = \psi(g, \psi(f,M)).\]
\end{proposition}

\begin{proof} It suffices to show that
\[ [ \  f(M)\ g(M)\  ]_Q = [ \ f(M) \ ]_Q \ [ \ g(\psi(f,M)) \ ]_Q .\]
Now, from the uniqueness of the $QR$ decomposition, for $X$ with $\det X >0$,
\[ [ \ QX\ ]_Q = Q [ \ X\ ]_Q \hbox{  and  }
 [ \ XR\ ]_Q = [ \ X\ ]_Q, \hbox{  for any }   Q \in SO(n), \ R \in \Uu^+(n) . \]
Using that $f(PXP^{-1})= P f(X) P^{-1}$ and the first equality just above,
\[ [ \ g(\psi(f,M)) \ ]_Q = [ \  {g \big([ \  f(M))\ ]_Q^T \ M \ [ \  f(M))\ ]_Q} \big)\ ]_Q= \]
\[
\big[ \ [ \ {f(M)}\ ]_Q^T \ {g( M )} \ [ \ {f(M)}\ ]_Q\ \big]_Q =
[ \ {f(M)}\ ]_Q^T \big[ \  {g( M )} \ [ \ {f(M)}\ ]_Q\ \big]_Q. \]
Go back to the beginning and use the second equality:
\[ [ \ f(M)\ ]_Q \ [ \  g(\psi(f,M)) \ ]_Q = \big[ \  {g( M )} \ [ \ {f(M)}\ ]_Q\ \big]_Q = \]
\[\big[ \  {g( M )} \ [ \ {f(M)}\ ]_Q\ [ \ {f(M)}\ ]_R\ \big]_Q
=  \big[ \  {g( M )} \  {f(M)} \big]_Q  =  [ \  {(fg)( M )}  ]_Q. \]
\end{proof}

The underlying flow should now be evident: this is the content of the next result.

\begin{corollary} Let $M, f, {\mathcal I}_M = \cO_M^O \cap \cO_M^U$ and $\psi$ as above. The map
\[ t, \ X \in \RR \times {\mathcal I}_M \ \mapsto \ \phi_f(t, X) := \psi( e^{tf}, X)  \quad \]
is a globally defined flow on ${\mathcal I}_M$. Flows associated to functions $f$ and $g$ commute.
\end{corollary}


\subsection{Diferential equations for the flows} \label{factorizations}

Take derivatives of  $\phi_f(t, M)$ to obtain the underlying differential equation. Set $Q(t) = [ \ e^{ f(M)}\ ]_Q$, $R(t) = [ \ e^{ f(M)}\ ]_R$ and $M(0)=M$.
\[ M(t) = \phi(t \ , M) = Q(t)^T \ M(0)\  Q(t) = R(t) \ M(0) \ R^{-1}(t)  \quad (\ast) \]
where $ e^{t f(M(0))} = Q(t) \ R(t)$,
so that, dropping time dependence,
\[ M'= (Q^T)' \  M(0) \ Q + Q^T M(0) \ Q'\ , \]
\[f(M(0))\  e^{t f(M(0))}  =  f(M(0)) \ Q \ R\ =  Q'\  R + Q \ R'. \]
To get a differential equation, eliminate $M(0) = Q(t) M(t) (Q(t))^T$:
\[ M'= (Q^T)'  Q M  +  M Q^T \ Q' \quad \hbox{and} \quad
Q \ f(M(0)) \ R\ = Q' R + Q \ R'. \]
Since $Q(t)$ is a curve of orthogonal matrices, $A(t) = Q^T \ Q'$ is a curve of skew symmetric matrices. Thus
\[ M'= - A M  +  M A = [M,A] \quad \hbox{and} \quad
 f(M) \  = Q^T \ Q' + \ R'R^{-1}. \]
 We are almost done: the matrix $f(M)$ in the last equation is a sum of a skew symmetric $Q^T \ Q'$ and and upper triangular $R'R^{-1}$, which bring to mind the factorization $M = \Pi_{sk}\,  M + \Pi_{up} M$ from Section \ref{coadjoint}.  Here $\Pi_{sk}\,  M$ and $M$ have the same strictly lower triangular part, which determines $\Pi_{sk}\,  M$. Adding up,
 \[ M'(t) = [\ M(t), A(t)\  ] = [\ M(t),\  \Pi_{sk}\,  \ f(M(t)) \ ]\ . \]
The case $f(x) = x$ is the Toda lattice after Flaschka's change of variables. Equation $(\ast)$ is Symes's  {\it solution by factorization} to the differential equation (\cite{S1}).

The fact that the Toda flow admits two different formulas by factorization, given by orthogonal and upper triangular conjugation of the initial condition, immediately implies that the evolution preserves the real tridiagonal symmetric form of the initial condition. Indeed, orthogonal conjugations preserve symmetry, and upper triangular conjugations preserve the upper Hessenberg form (i.e., the only nonzero entries below the diagonal lie in the subdiagonal of entries
$(i+1,i)$).

\section{Toda, QR and other algorithms}

In the fifties, Francis \cite{Fr} came up with the $QR$ algorithm to compute eigenvalues of symmetric matrices. Say $S = S_0$ is a real, symmetric matrix of positive simple spectrum. Consider the alternation of $QR$ decompositions and reorderings,
\[ S_0 = Q_0 R_0, \quad S_1 \, = \,  R_0 Q_0 \,  = \, Q_0^T\, S_0 \, Q_0 \, = \, R_0 \, S_0 \,  R_0^{-1}. \]

Clearly, $S_1 \in \cO_{S_0}$, so $S_0$ and $S_1$ are both symmetric with a common spectrum: the {\it $QR$ step} is a diffeomorphism from  $\cO_{S_0}$ to itself. The remarkable thing about it is that iteration of this map converges to a diagonal matrix $\Lambda$ --- since a $QR$ step clearly preserves spectrum, the diagonal entries of $\Lambda$ are the eigenvalues of $S_0$!

It was Moser, again, who drew Deift's attention to Symes's beautiful connection between
the Toda flow and the $QR$ algorithm \cite{S2}: at integer times $n$, the solution $J(t)$ of the
Toda differential equation satisfies $exp \ J(n) = E_n$, where $E_n$ is the $n$-th term in the
QR sequence starting from $E_0 = exp J_0$.
On the other hand, since the Toda equation is one within a family of flows of the kind $\dot{J} = [ f, \Pi_{skew} f(J)]$, one may fudge with the functional parameter (\cite{DNT}) and get a simpler relationship.

\begin{theorem}  The flow associated to  $\dot{S} = [ f, \Pi_{skew} \ln S]$ interpolates the $QR$ sequence $S_n$, where $S_0 = S(0)$. \end{theorem}

\begin{proof}
From Section \ref{factorizations}, the solution of the equation is given by
\[ S(t) = Q^\ast(t) S(0) Q(t), \quad  \hbox{ where } \exp t\ln S(0) = Q(t) R(t).\]
Thus, $S_0= J(0) = Q(1) R(1)$ and $S(1) = Q^\ast(1) S(0) Q(1) = J_1,$ the matrix obtained from $S_0$ from a $QR$ step.
\end{proof}

Since a Toda flow interpolates the $QR$ iteration, the convergence properties of Toda flows are also satisfied by the iteration.  The vocabulary of dynamical systems clarifies certain eigenvalue computations. The right side of
Figure \ref{polytopes} represents $\bar{{\mathcal J}_\Lambda}$,  the closure of the set of $4 \times 4$ Jacobi matrices with eigenvalues $1, 2, 3$ and $4$. The vertices correspond to diagonal matrices, which are equilibria for the Toda vector field. The vertex
$I$ is a source, $A$ is a sink and the remaining vertices are saddles with different
signatures (i.e., dimensions of the unstable manifold). The presence of saddles explains why orbits bifurcate close to some vertices (say, $H$, $E$ and $F$) in the neighborhood of which an orbit spends a long time (i.e., many $QR$ iterations), a fact that was known in the numerical literature as \textit{root disorder}.

Numerical analysts might have realized a long time ago that the $QR$ algorithm is the integer evaluation of a flow. Indeed, it has been known for decades that one can obtain directly the matrix $S_n$ of the $QR$ iteration starting with a symmetric matrix $S_0$ is given by $S_n = Q^\ast_n S_0 Q_n = R_n S_0 R^{-1}_n$ where $Q_n$ and $R_n$ are obtained from the $QR$ factorization $S_0^n = Q_n R_n$. Morally (and indeed correctly), the step
\[ S_{1/n} = Q^\ast_{1/n} S_0 Q_{1/n} = R_{1/n} S_0 R^{-1}_{1/n}, \quad S_0^{1/n} = Q_{1/n} R_{1/n} \]
is an $n$-th root of the usual $QR$ step (in the sense that $n$ such steps yield the usual $QR$ step). Now, to obtain the interpolating flow (which belongs to the Toda family, as we saw) simply compute  \[ \lim_{n \to \infty} n (S_{1/n}-S_0) = \lim_{n \to \infty} n(Q^\ast_{1/n} S_0 Q_{1/n}-S_0).   \quad (\ast) \]
Up to order $1/n$,
\[ S_0^{1/n} \simeq I + \frac{\ln S_0}{n}, \quad I + \frac{\ln S_0}{n} \simeq  (I+ \Pi_{sk} \frac{\ln S_0}{n}) (I+ \Pi_{up} \frac{\ln S_0}{n}) \]
and thus $Q_{1/n} \simeq I+ \Pi_{sk} \frac{\ln S_0}{n}$. We are now ready to take the limit $(\ast)$: the interpolating flow, for which $f(x) = \ln (x)$, is
\[ S'(t) = [ \, S, \Pi_{sk}\,  \ln S \,]. \]

\subsection{Choleski and SVD}

There is nothing special about $QR$ factorizations: other factorizations give rise to flows which are very similar to the Toda lattice.

A \textit{Cholesky} factorization $M = M_L M_U$ decomposes a matrix $M$ as a product of lower and upper triangular matrices $M_L$ and $M_U$ with the same positive diagonal. The factorization can be performed (uniquely) for matrices in $G_+$, having upper principal minors with positive determinant.
Now (\cite{DLT1}), on $G_+$ define the product $g * h = h_L g h_U$. For an appropriate coupling,
the induced dual Lie algebra is the phase space which accommodates the Cholesky iteration,
$$M_n = L_n U_n, \quad M_{n+1} = U_n L_n,  \quad n=0,1,\ldots$$
and its continuous interpolation (notice that blowups may happen).
The Lie bracket associated to this group structure
is an example of the so called \textit{$R$-matrix formalism} applied to the standard matrix Lie bracket, but this is another story.

Given a real matrix $M$, its {\it singular values} are the (nonnegative) lengths of the semi-axis of the ellipsoid obtained by applying $M$ to the unit (Euclidean) sphere. A {\it singular value decomposition} of $M$ is a product $M = Q \Sigma U$, where $Q$ and $U$ are orthogonal matrices and $\Sigma$ is a diagonal matrix having the singular values as diagonal entries.

There is an efficient algorithm, similar to  $QR$,  to compute singular values of tridiagonal matrices, which was shown by Demmel and  Kahan to have remarkable stability  properties with respect to relative errors (\cite{DK}).  In \cite{DDLT},
these properties were studied from a symplectic setup: the appropriate phase space is chosen taking into account the
specific concern with relative errors. The Jacobian $M(i,j)$ of the map sending a matrix at step $i$
to its (discrete) evolution at time $j$ is analyzed using Krein's perturbation theory for symplectic
matrices. A number of properties arise, which are responsible for the good performance of the algorithm: for large $i$, the spectrum of $M(i,j)$ is simple, lies
in the unit circle and $M(i,j)$ converges to a limit, explicitly computed. The agreement between theoretical estimates and experiments is remarkable: the rate between computed and estimated error was never larger than 8, independent of dimension.

\section{Bidiagonal coordinates} \label{bidiagonal}

Norming constants in Section \ref{LaxMoser} have a drawback: they do not cover the limit points of algorithms which converge to reduced matrices (i.e., tridiagonal matrices with some main off-diagonal entries equal to zero), like the Toda flows and $QR$ type algorithms. This problem has been circumvented by the introduction of {\it bidiagonal coordinates} (\cite{LST1}). As an extra bonus, bidiagonal coordinates provide an atlas for $\cT_\Lambda$. The construction goes as follows.

Let $\cL^1(n)$ denote the group of lower triangular matrices with unit diagonal entries. For $M \in \cM(n)$, the {\it $LU$ positive factorization}, when it exists, is $M = L U$, where $L \in \cL^1(n)$ and $U \in \cU^+(n)$. Clearly, this happens if and only if the determinants of the upper principal minors of $M$ are positive (more on the appendix).

Let $S_n$ be the symmetric group on the set $\{1,2, \ldots,n\}$.
For $\ \pi \in S_n$, define
\[ \Lambda_\pi = \diag(\lambda_{\pi(1)}, \lambda_{\pi(2)},\ldots, \lambda_{\pi(n)}). \]

There is one chart of $\cT_\Lambda$ for each permutation $\pi$. Each chart has for domain the set  $\UpLa \subset\ILa \ $ consisting of matrices $T = Q_\pi^\ast \,\Lambda_\pi Q_\pi$ for which there exists an orthogonal matrix $Q_\pi$ admitting an $LU$ positive factorization $Q_\pi = L_\pi U_\pi$. This is not as restrictive as it looks: if $T = Q_\pi^\ast \,\Lambda_\pi Q_\pi$, another spectral decomposition is given by $T = Q_\pi^\ast E \,\Lambda_\pi E Q_\pi$, where $E$ is a diagonal matrix with entries equal to $\pm 1$ along the diagonal: one may use $E$ to force the positivity of the determinants of the principal minors of $E Q_\pi$, provided that the corresponding determinants of $ Q_\pi$ are nonzero. Notice that the request that $Q_\pi$ admits an $LU$ positive factorization gives rise to a \emph{unique} spectral decomposition $T = Q_\pi^\ast \,\Lambda_\pi Q_\pi$.
We then have
\[ T = Q_\pi^\ast \,\Lambda_\pi Q_\pi = (L_\pi U_\pi)^{-1} \,\Lambda_\pi L_\pi U_\pi =  U_\pi^{-1} B_\pi  U_\pi \]
where
\[ B_\pi = L_\pi^{-1} \,\Lambda_\pi L_\pi  = U_\pi T U_\pi^{-1}. \]

\begin{theorem} \label{beta}
The matrix $B_\pi$ is lower bidiagonal with diagonal $\Lambda_\pi$.
The principal off-diagonal entries $\beta^\pi_k = (B_\pi)_{k+1,k}$
define a diffeomorphism $\psi_\pi: \UpLa \to \RR^{n-1}$.
Each  domain $\UpLa$ is an open, dense set $\UpLa$, containing one diagonal matrix. The charts $\{ \psi_\pi: \UpLa \to \RR^{n-1}, \pi \in S_n\}$ form an atlas for $\cT_\Lambda$.
The signs of $\beta^\pi_k$ and $T_{k+1,k}$ are equal and their quotient goes to one, when one of them goes to zero.
\end{theorem}

\begin{proof} From the expressions for $B_\pi$, it is simultaneously lower triangular and upper Hessenberg. So it is actually
lower bidiagonal, with the same spectrum as $\Lambda_\pi$.

To show that the chart is a diffeomorphism to $\RR^{n-1}$, consider the construction of its inverse. Build $B_\pi$ out of off-diagonal entries $\beta^\pi_j$ and (distinct) eigenvalues $\lambda_{\pi(i)}$.
Diagonalize $B_\pi = L_\pi^{-1} \,\Lambda_\pi L_\pi$ and get $Q_\pi$ out of the $QR$ factorization $L_\pi = Q_\pi R_\pi$, so that, automatically, $Q_\pi$ admits an $LU$ positive factorization, hence $Q_\pi \in \UpLa$. Finally set $T = Q_\pi^\ast \,\Lambda_\pi Q_\pi$.

Since $U \in \cU^+(n)$, the equation $B_\pi = U_\pi T U_\pi^{-1}$ gives that the signs of $\beta^\pi_k$ and $T_{k+1,k}$ are equal.
The remaining statements are left to the reader.
\end{proof}

Figure \ref{ilailah} is an example for $n=3$: here $\Lambda_\pi = \diag(7,4,5)$ and $\UpLa$ is the interior of the polygon with boundary given by the unglued edges (the glued edges belong to $\UpLa$). Figure \ref{bitorus} was obtained using such charts: the standard norming constants would distort too much the picture (and degenerate completely) at boundaries of (signed) Jacobi matrices. Once eigenvalues are fixed, the bitorus $\cT_\lambda$ lies in the intersection of a hyperplane (of matrices with the same trace as $S$) and a sphere (same sum of squared eigenvalues), and the figure is the image of a conformal projection of $\cT_\lambda$ in $\RR^2$.

Another remarkable property of bidiagonal coordinates is that their evolution under the Toda equations manages to be even simpler that the evolution of the standard inverse variables. We consider the Toda vector fields on real, symmetric, tridiagonal matrices, where the time dependence is explicit,
\[ T'(t) = [ T(t),\Pi_{sk}\,  ( f(T(t)))]\, . \eqno(T) \]

\begin{proposition} \label{evolutions}
Fix $\pi$, take $T(0) \in \UpLa$. The chart domain $\UpLa$ is invariant under equation (T). The evolution of the bidiagonal coordinates is
\[ B_\pi'(t) = [B_\pi(t), - f(\Lambda_\pi) ], \hbox{ i.e., }  (\beta^{\pi}_i)'(t)  =
   (f(\lambda_{\pi(i+1)}) - f(\lambda_{\pi(i)})) \ \beta^{\pi}_i(t) \ . \]
  \end{proposition}
\begin{proof}  We first prove that $(T)$ leaves $\UpLa$ invariant. Take $T(0) \in \UpLa$:  omitting the permutation $\pi$, we have $T(0) = Q^\ast_0 \, \Lambda \,Q_0$, where $Q_0$ has an $LU$ positive decomposition. Solve $(T)$  as in Section  \ref{factorizations}:
for $\exp(t f(T(0))) = Q(t) \,R(t)$,
\[ T(t) = Q(t)^\ast \, T(0)\, Q(t) = Q(t)^\ast \, Q^\ast_0 \,\Lambda \, Q_0 \, Q(t)  \]
We have to prove that $\tilde{Q}(t)= Q_0 \, Q(t)$ admits an $LU$ positive factorization:
\[ \tilde{Q}(t)=  Q_0 \, \exp(t f(T(0)))\, R^{-1}(t)
=  Q_0 \, \exp(t f(T(0)))\,Q_0^\ast \, Q_0 \, R^{-1}(t) \]
\[ = Q_0 \, \exp(t f(Q^\ast_0 \, \Lambda \,Q_0))\,Q_0^\ast \, Q_0 \, R^{-1}(t) = \exp(t f(\Lambda))\, Q_0 \, R^{-1}(t).\]
The  upper principal minors of $\tilde{Q}(t)$ and $Q_0$ have the same signs, since  $\exp(t f(\Lambda))$ is a positive diagonal matrix and
$R^{-1}(t) \in \cU^+(n)$.

We now consider bidiagonal coordinates. Clearly
\[ T(t) = Q^\ast(t) Q^\ast_0 \,\Lambda\, Q_0 Q(t) = \tilde{Q}^\ast(t) \,\Lambda\, \tilde{Q}(t), \quad \tilde{Q}^\ast(t) \, \tilde{Q}'(t)= Q^\ast(t) Q'(t).\]

Now, $B(t) = L^{-1}(t) \,\Lambda\, L(t)$ and we obtain $B'(t) = [ B(t), L^{-1}(t) \, L'(t)]$ by now familiar computations. The matrix $L$ is obtained by the $LU$ positive factorization $\tilde{Q}(t) = L(t) \, U(t)$, so that, dropping the time dependence, $\tilde{Q}' = L'U + LU'$ yields
\[ L^{-1} \tilde{Q}' U^{-1} = L^{-1} L' + U'U^{-1}. \]
Consider the split $M = \Pi_{sl}\, M + \Pi_{u} M$ of a matrix $M$ into strictly lower and upper triangular parts. Since $L$ has a diagonal of ones, $L^{-1} L'$ is strictly lower triangular.
The solution by factorization in Section \ref{factorizations} gives
$ \tilde{Q}^\ast \, \tilde{Q}' = \Pi_{sk}\,  ( f(T))$, so that
\[ L^{-1} L' = \Pi_{sl}\,  L^{-1} \tilde{Q}' U^{-1} = \Pi_{sl}\,  U \big( \Pi_{sk}\,   f(T) \big) \, U^{-1}. \]
Now, the matrices $M$ and $\Pi_{sk}\,  M$ have the same strictly lower triangular part and $U$ is upper triangular, so
\[  L^{-1} L'  = \Pi_{sl}\,  U   f(T)  U^{-1} = \Pi_{sl}\,    f(UTU^{-1})  = \Pi_{sl}\,    f(B) . \]
Adding up,
\[ B' = [ B, L^{-1} L'] = [B, \Pi_{sl}\,    f(B)] = [B,  f(B) - f(\Lambda)] =  [B, - f(\Lambda)] \]
\end{proof}

The proof requires interpretation in the case $f(x) = \ln x$, which is of relevance for $QR$  interpolation: we need $\exp(t f(\Lambda))$ to be a positive diagonal matrix --- in this case, the diagonal entries must be equal to $|\lambda_k|$.

Bidiagonal coordinates are especially convenient to study asymptotic behavior of Toda flows. As an application, the reader may find in \cite{LST1} a rather natural computation of their scattering map, described in Section \ref{LaxMoser}, by filling up the following  inevitable outline. Recall from Figure 1 that Toda flows starting from Jacobi matrices have for $\omega$ and $\alpha$ limits the diagonal matrices associated to
$ \pi_\omega (i) = i$ and $\pi_\alpha (i) = n-i+1$.
On each chart, the Toda evolution in bidiagonal coordinates is simple. The change of charts required to keep track of both extremes of an orbit is equally simple.

An appropriate extension of this formalism provides charts on
isospectral manifolds of real and complex matrices with given {\it profile}, a natural extension of the concept of tridiagonality --- a text is under preparation (\cite{ST}).

\subsection{$QR$ steps with shifts}

Numerical analysts have  ways to speed up the convergence of $QR$. The original $QR$ step is just the choice $f(x) = \ln x$ in the family
\[ S_n \mapsto  S_{n+1} = [ \exp f(S_0) ]_Q^\ast \, S_0 \,[ \exp f(S_0) ]_Q = [ \exp f(S_0) ]_R \,S_0 \, [\exp f(S_0) ]_R^{-1}.\]
Bidiagonal coordinates may be used to indicate interesting alternatives:  interpret a step as a time 1 map for a differential equation and integrate the trivial flow which describe evolutions in  Proposition \ref{evolutions}. For $f(x) = \ln g(x)$, taking into account the caveat after its proof, the change of the bidiagonal coordinates of a matrix $T \in \UpLa$   under a step is
\[ (\beta^\pi_1, \ldots, \beta^\pi_{n-1}) \mapsto
\left( \left | \frac{g(\lambda_{\pi(2)})}{g(\lambda_{\pi(1)})} \right|
\beta^\pi_1, \ldots,
\left | \frac{g(\lambda_{\pi(n)})}{g(\lambda_{\pi(n-1)})} \right|
\beta^\pi_{n-1} \right). \]
We are interested in functions $g$ for which some $\beta^\pi_{k}$ becomes small --- in this case, the corresponding matrix essentially decouples in two smaller tridiagonal matrices, for which the computation of eigenvalues is simpler. Taking into account the denominators in the formula, a natural possibility is a function $g$ which equals zero at $\lambda_{\pi(n)}$ but does not vanish at the other eigenvalues. Keep in mind that, unfortunately, we do not know the eigenvalues. In particular, it is especially hard to obtain such a $g$ which would reduce drastically a centrally located $\beta^\pi_{k}$ (i.e., $k \sim n/2$).

Notice that, from Theorem \ref{beta}, the matrix $T$ is Jacobi if and only if its bidiagonal coordinates $\beta^\pi_k$ are positive, and the sign is preserved under general $QR$ steps. On the other hand, removing the absolute values in the formula shows that one can replace iterations lying within Jacobi matrices to iterations on tridiagonal matrices without changing their asymptotic properties: only the off-diagonal entries eventually have different signs. The new iteration is {\it smooth} on the isospectral manifold $\cT_\Lambda$, and convergence rates may be obtained using Taylor expansions.

Typically, along an iteration of an algorithm searching for eigenvalues, one has good approximations $s$ for one of them: a natural choice is $g(x) = x - s$, the iteration with {\it shift} $s$. The bidiagonal coordinates, in this case, change as follows:
\[
(\beta^\pi_1, \ldots, \beta^\pi_{n-1}) \mapsto
\left( \left|\frac{\lambda^\pi_2 - s}{\lambda^\pi_1 - s}\right| \beta^\pi_1,
\ldots,
\left|\frac{\lambda^\pi_n - s}{\lambda^\pi_{n-1} - s}\right| \beta^\pi_{n-1}
\right). \]

Numerical analysts frequently do not wait for a good approximation $s$. There are different shift strategies (\cite{P}) --- we consider the {\it  Rayleigh quotient shift} for which $s = T_{n,n}$,
and the {\it Wilkinson shift}: compute the eigenvalues of $\hat{T}$, the bottom $2 \times 2$ principal minor of $T$, and take for $s$ the one which is closer to $T_{n,n}$.

Under the Rayleigh shift, once iteration approaches convergence, a simple Taylor expansion shows that the bottom entry $b = \beta^\pi_{n-1}$
converges cubically to zero, in the sense that $b^{m+1} = O((b^m)^3)$ for bottom entries at consecutive iterations. Once $b$ is small enough, the matrix undergoes {\it deflation}: $T_{n,n}$ is declared a good approximation of an eigenvalue and the algorithm proceeds with the top $(n-1) \times (n-1)$ block. There is one catch however: for some initial conditions, the Rayleigh shift gives rise to periodic orbits.

This does not happen for the Wilkinson shift. The dynamics in this case is richer: we describe the results but may only indicate \cite{LST2} and \cite{LST3} for proofs.

\begin{theorem} For a generic initial condition, a Wilkinson iteration leads to cubic convergence of $b = \beta^\pi_{n-1} \sim T_{n,n-1}$. If the original matrix $T$ has no three eigenvalues in arithmetic progression, this is always the case. Otherwise, there may be a Cantor-like set of initial conditions for which iteration is quadratic.
\end{theorem}

In a nutshell, this peculiar behavior is caused by the discontinuities in the definition of $s$, at points where $T_{n,n}$ is equidistant from the eigenvalues of the block $\hat T$. The Cantor-like set in the statement of the theorem consists of matrices all of whose iterations lie in this situation. Typically, this never happens or happens just at few steps and convergence is cubic.

\section{Lax pairs beyond Toda}
Forty years of contributions from
a large community greatly increased our understanding of the Toda flow, which sometimes is undistinguishable from the more general Lax pair evolution. Mutations are abundant: for a beautiful starting point about different representations of the Toda flow (and other geometric issues which may studied through them), the reader should refer to \cite{KS}.

In this text, we emphasized examples over theory. In order to indicate the versatility of the concepts which have been presented, we close with a final example.

\subsection{The billiard on an ellipsoid}

Moser and Veselov \cite{MV} introduced a very interesting formulation of the  billiard on an ellipsoid $E = \{ x : (x, C^{-2}x) \leq 1\}$. We fix notation: a ball moves along a direction $y_0$ (a unit vector), hits the boundary $\partial E$ at $x_0$, and
leaves in direction $y_1$ until it hits $x_1$. They presented the \textit{billiard map} $\Psi(x_0,y_0) = (x_1,y_1)$
in the following fashion. Write the (matrix) polynomial
\[ L_0(\lambda) = y_0 \otimes y_0 + \lambda x_0 \wedge y_0 - \lambda^2 C^2, \]
factor in linear terms,
$$L_0(\lambda) = (\lambda C + y_0 \otimes \xi_0)(-\lambda C + \xi_0 \otimes y_0), \quad \xi_0 = C^{-1} x_0,\, ||\xi_0||=1, $$
intertwine factors as in the $QR$ algorithm and factor again,
$$L_0'(\lambda) =(-\lambda C + \xi_0 \otimes y_0)(\lambda C + y_0 \otimes \xi_0)
=(\lambda C + y_0' \otimes \xi_0')(-\lambda C + \xi_0' \otimes y_0'),$$
with $y_0'= \psi_0, ||\xi_0'||=1$. Set $\phi(x_0,y_0) = (-Cy_1,C^{-1}x_0)$: then
$\Psi = - \phi^2$.

In \cite{DLT2}, the analogy of the Moser-Veselov formalism to the $QR$ iteration
--- better still, to the Cholesky iteration \cite{DLT1} --- is taken literally. The relevant group now is $G_+$, consisting of loops
$\gamma: i\RR \to GL(n,\CC)$ which are smooth at $\infty$, contract to the identity loop,
are positive diagonal at $\infty$,  satisfy the reality condition
 $\gamma(\bar{\lambda}) = \overline{{\gamma(\lambda)}}$ and
admit a (unique) Riemann-Hilbert factorization $\gamma = \gamma_L \gamma_R$,
where $\gamma_L$ and $\gamma_R$ have analytic continuations to the left and right side of the imaginary axis
respectively, and same diagonal values at $\infty$.
The group operation is
$\gamma * \delta = \delta_L \gamma \delta_R$.
Take an initial data for the billiard evolution and construct
the rational matrix function $A_0(\lambda) = L_0(\lambda)/(1 - \lambda^2)$.

One has to circumvent a technical difficulty:  $A_0(0)$ is not invertible --- it is a rank one matrix!
Still, as in the finite dimensional cases,
the evolution is interpolated by a differential equation. This
time, the solution formula by factorization involves a {\it Riemann-Hilbert problem}, with a mild singularity at $\lambda=0$.

\section*{Appendix 1: The $QR$ and $LU$ factorizations}

Let $M$ an invertible, real matrix. Then there is a unique $Q \in SO(n)$ and $R \in \Uu(n)$ for which $M = QR$, the {\sl $QR$ factorization of $M$}. Indeed, let $e_i, i=1,\ldots,n$ be the canonical vectors: the reader should have no difficulty in showing that the subspaces generated by $Me_i$ and $Q e_i$, for $i=1,\ldots,k$ ($k$ arbitrary) should be the same, if such a decomposition exists. Since the columns of $Q$ are orthonormal, they must be  vectors obtained by applying the Gram-Schmidt orthogonalization procedure to the columns of $M$ sequentially. The entries of $R$ are the coefficients used in the representation of $Me_i$ in terms of the columns $Qe_1, \ldots, Qe_i$. Since $M$ is invertible, the process is feasible, and appropriate normalizations give rise to positive diagonal entries of $R$.

Similarly, the (unique) $LU$  decomposition of $M$ is $M = LU$, where $L \in \cL^1 (n)$ and $U \in \cU^+(n)$. This decomposition exists if and only if the principal diagonal minors of $M$ (i.e., the determinants of the $k \times k$ submatrices $M_k$ with entries in the intersection of the first $k$ rows and columns, $k=1,\ldots,n$) are strictly positive. In a nutshell, from $M_k = L_k U_k$, it is clear that an inductive construction is at hand: the details are left to the reader or in standard texts (\cite{De}, \cite{TB}).

\medskip
Received June 2013; revised October 2013.
\medskip

\end{document}